\documentclass[a4paper,10pt]{amsart}
\usepackage{hyperref}
\usepackage{ifluatex}
\usepackage{mathrsfs}
\usepackage{tikz}
  \usetikzlibrary{arrows,%
    decorations.markings,%
    matrix,%
    shapes}

\ifluatex
  \usepackage{fontspec}
  \usepackage{unicode-math}
 \else
  \usepackage{fourier}
\fi

\hypersetup{pdftitle={The morphism axiom for n-angulated categories},%
  pdfauthor={Emilie Arentz-Hansen, Petter Andreas Bergh, and Marius
    Thaule},%
  pdfkeywords={Triangulated categories, $n$-angulated categories,
    morphism axiom}
}

\DeclareMathOperator{\Hom}{Hom}
\newcommand{\C}{\mathcal{C}}
\newcommand{\ceq}{\mathrel{\mathop:}=}
\newcommand{\diagram}[3]{\matrix (#1) [matrix of math nodes,row
  sep={#2},column sep={#3},text height=1.5ex,text depth=0.25ex]}
\newcommand{\nang}{\mathscr{N}}

\theoremstyle{definition}
\newtheorem{remark}[subsection]{Remark}

\theoremstyle{plain}
\newtheorem{theorem}[subsection]{Theorem}

\title{The morphism axiom for $n$-angulated categories}
\author{Emilie Arentz-Hansen}
\author{Petter Andreas Bergh}
\author{Marius Thaule}
\address{Department of Mathematical Sciences, NTNU, NO-7491
  Trondheim, Norway}
\date{\today}
\email{emilieba@stud.ntnu.no}
\email{bergh@math.ntnu.no}
\email{mariusth@math.ntnu.no}

\begin{document}

% Abstract
\begin{abstract}
  We show that the morphism axiom for $n$-angulated categories is
  redundant.
\end{abstract}


\maketitle

\section{Introduction}
\label{sec:intro}

In \cite{GKO}, Geiss, Keller and Oppermann introduced $n$-angulated
categories. These are generalizations of triangulated categories, in
the sense that triangles are replaced by $n$-angles, that is, morphism
sequences of length $n$. Thus a $3$-angulated category is precisely a
triangulated category.

There are by now numerous examples of such generalized triangulated
categories. They appear for example as certain cluster tilting
subcategories of triangulated categories. Recently, in \cite{Jasso},
Jasso introduced the notion of \emph{algebraic} $n$-angulated
categories. Such a category is by definition equivalent to a stable
category of a Frobenius $n$-exact category, the latter being a
genralization of an exact category; exact sequences are replaced by
so-called $n$-exact sequences. Algebraic $n$-angulated categories are
therefore natural generalizations of algebraic triangulated
categories, which are defined as categories equivalent to stable
categories of Frobenius exact categories. By \cite[Section
6.5]{Jasso}, the $n$-angulated categories that arise as cluster tilted
subcategories of algebraic triangulated categories are themselves
algebraic. Actually, the only examples so far of $n$-angulated
categories that are not algebraic are the rather exotic ones in
\cite{BerghJassoThaule}. At the time of writing, there is no notion of
\emph{topological} $n$-angulated categories, except in the
triangulated case.

The axioms defining $n$-angulated categories are generalized versions
of the axioms defining triangulated categories. Among these axioms,
the morphism axiom and the generalized octahedral axiom are the ones
that guarantee an interesting theory. Apart from the existence axiom,
which states that every morphism is part of an $n$-angle, the other
axioms are ``bookkeeping axioms,'' to borrow Paul Balmer's
terminology. The generalized octahedral axiom was introduced in
\cite{BerghThaule}, but in this paper we present it in a much more
compact and readable form.

The topic of this paper is the morphism axiom, which states that a
morphism between the bases of two $n$-angles can be extended to a
morphism of $n$-angles. We show that this axiom is redundant; it
follows from the generalized octahedral axiom and the existence
axiom. For triangulated categories, this was proved by May in
\cite{May}.

\section{The axioms for $n$-angulated categories}
\label{sec:axioms}

In this section, we recall the axioms for $n$-angulated categories. We
fix an additive category $\C$ with an automorphism
$\Sigma \colon \C \to \C$, and an integer $n$ greater than or equal to
three.

A sequence of objects and morphisms in $\C$ of the form
\begin{equation*}
  A_1 \xrightarrow{\alpha_1} A_2 \xrightarrow{\alpha_2} \cdots
  \xrightarrow{\alpha_{n - 1}} A_n \xrightarrow{\alpha_n} \Sigma A_1 
\end{equation*}
is called an \emph{$n$-$\Sigma$-sequence}. Whenever convenient, we
shall denote such sequences by $A_\bullet, B_\bullet$ etc. The left
and right \emph{rotations} of $A_{\bullet}$ are the two
$n$-$\Sigma$-sequences
\begin{equation*}
  A_2 \xrightarrow{\alpha_2} A_3 \xrightarrow{\alpha_3} \cdots
  \xrightarrow{\alpha_n} \Sigma A_1 \xrightarrow{(-1)^n
    \Sigma\alpha_1} \Sigma A_2
\end{equation*}
and
\begin{equation*}
  \Sigma^{-1}A_n \xrightarrow{(-1)^n \Sigma^{-1} \alpha_n} A_1
  \xrightarrow{\alpha_1} \cdots \xrightarrow{\alpha_{n - 2}} A_{n-1}
  \xrightarrow{\alpha_{n - 1}} A_n
\end{equation*} 
respectively, and a \emph{trivial} $n$-$\Sigma$-sequence is a sequence
of the form
\begin{equation*}
  A \xrightarrow{1} A \to 0 \to \cdots \to 0 \to \Sigma A
\end{equation*}
or any of its rotations. 

A \emph{morphism} $A_{\bullet} \xrightarrow{\varphi} B_{\bullet}$ of
$n$-$\Sigma$-sequences is a sequence
$\varphi = (\varphi_1,\varphi_2,\ldots,\varphi_n)$ of morphisms in
$\C$ such that the diagram
\begin{center}
  \begin{tikzpicture}
    \diagram{d}{2.5em}{2.5em}{
      A_1 & A_2 & A_3 & \cdots & A_n & \Sigma A_1\\
      B_1 & B_2 & B_3 & \cdots & B_n & \Sigma B_1\\
    };
    
    \path[->,midway,font=\scriptsize]
      (d-1-1) edge node[above] {$\alpha_1$} (d-1-2)
                   edge node[right] {$\varphi_1$} (d-2-1)
      (d-1-2) edge node[above] {$\alpha_2$} (d-1-3)
                   edge node[right] {$\varphi_2$} (d-2-2)
      (d-1-3) edge node[above] {$\alpha_3$} (d-1-4)
                   edge node[right] {$\varphi_3$} (d-2-3)
      (d-1-4) edge node[above] {$\alpha_{n - 1}$} (d-1-5)
      (d-1-5) edge node[above] {$\alpha_n$} (d-1-6)
                   edge node[right] {$\varphi_n$} (d-2-5)
      (d-1-6) edge node[right] {$\Sigma\varphi_1$} (d-2-6)
      (d-2-1) edge node[above] {$\beta_1$} (d-2-2)
      (d-2-2) edge node[above] {$\beta_2$} (d-2-3)
      (d-2-3) edge node[above] {$\beta_3$} (d-2-4)
      (d-2-4) edge node[above] {$\beta_{n - 1}$} (d-2-5)
      (d-2-5) edge node[above] {$\beta_n$} (d-2-6);
  \end{tikzpicture}
\end{center}
commutes.  It is an \emph{isomorphism} if
$\varphi_1,\varphi_2,\ldots,\varphi_n$ are all isomorphisms in $\C$,
and a \emph{weak isomorphism} if $\varphi_i$ and $\varphi_{i + 1}$ are
isomorphisms for some $1 \leq i \leq n$ (with
$\varphi_{n + 1} \ceq \Sigma \varphi_1$).  Note that the composition
of two weak isomorphisms need not be a weak isomorphism.  Also, note
that if two $n$-$\Sigma$-sequences $A_{\bullet}$ and $B_{\bullet}$ are
weakly isomorphic through a weak isomorphism
$A_{\bullet} \xrightarrow{\varphi} B_{\bullet}$, then there does not
necessarily exist a weak isomorphism $B_{\bullet} \to A_{\bullet}$ in
the opposite direction.

\sloppy Let $\nang$ be a collection of $n$-$\Sigma$-sequences in
$\C$. Then the triple $( \C, \Sigma, \nang )$ is an
\emph{$n$-angulated category} if the following four axioms are
satisfied:

\begin{itemize}
\item[{\textbf{(N1)}}]
  \begin{itemize}
  \item[(a)] $\nang$ is closed under direct sums, direct summands and
    isomorphisms of $n$-$\Sigma$-sequences.
  \item[(b)] For all $A \in \C$, the trivial $n$-$\Sigma$-sequence
   \begin{equation*}
     A \xrightarrow{1} A \to 0 \to \cdots \to 0 \to \Sigma A
   \end{equation*}
   belongs to $\nang$.
 \item[(c)] For each morphism $\alpha \colon A_1 \to A_2$ in $\C$,
   there exists an $n$-$\Sigma$-sequence in $\nang$ whose first
   morphism is $\alpha$.
 \end{itemize}
\item[{\textbf{(N2)}}] An $n$-$\Sigma$-sequence belongs to $\nang$ if
  and only if its left rotation belongs to $\nang$.
\item[{\textbf{(N3)}}] Given the solid part of the commutative diagram
  \begin{center}
    \begin{tikzpicture}
      \diagram{d}{2.5em}{2.5em}{
        A_1 & A_2 & A_3 & \cdots & A_n & \Sigma A_1\\
        B_1 & B_2 & B_3 & \cdots & B_n & \Sigma B_1\\
      };
    
      \path[->,midway,font=\scriptsize]
        (d-1-1) edge node[above] {$\alpha_1$} (d-1-2)
                     edge node[right] {$\varphi_1$} (d-2-1)
        (d-1-2) edge node[above] {$\alpha_2$} (d-1-3)
                     edge node[right] {$\varphi_2$} (d-2-2)
        (d-1-3) edge node[above] {$\alpha_3$} (d-1-4)
                     edge[densely dashed] node[right] {$\varphi_3$} (d-2-3)
        (d-1-4) edge node[above] {$\alpha_{n - 1}$} (d-1-5)
        (d-1-5) edge node[above] {$\alpha_n$} (d-1-6)
                     edge[densely dashed] node[right] {$\varphi_n$} (d-2-5)
        (d-1-6) edge node[right] {$\Sigma\varphi_1$} (d-2-6)
        (d-2-1) edge node[above] {$\beta_1$} (d-2-2)
        (d-2-2) edge node[above] {$\beta_2$} (d-2-3)
        (d-2-3) edge node[above] {$\beta_3$} (d-2-4)
        (d-2-4) edge node[above] {$\beta_{n - 1}$} (d-2-5)
        (d-2-5) edge node[above] {$\beta_n$} (d-2-6);
    \end{tikzpicture}
  \end{center}
  with rows in $\nang$, the dotted morphisms exist and give a morphism
  of $n$-$\Sigma$-sequences.
\item[{\textbf{(N4)}}] Given the solid part of the diagram
  \begin{center}
    \begin{tikzpicture}
      \diagram{d}{2.5em}{2.5em}{
        A_1 & A_2 & A_3 & A_4 & \cdots & A_{n - 1} & A_n & \Sigma A_1\\
        A_1 & B_2 & B_3 & B_4 & \cdots & B_{n - 1} & B_n & \Sigma
        A_1\\
        A_2 & B_2 & C_3 & C_4 & \cdots & C_{n - 1} & C_n & \Sigma
        A_2\\
      };

      \path[->,midway,font=\scriptsize]
        (d-1-1) edge node[above] {$\alpha_1$} (d-1-2)
        ([xshift=-0.1em] d-1-1.south) edge[-] ([xshift=-0.1em] d-2-1.north)
        ([xshift=0.1em] d-1-1.south) edge[-] ([xshift=0.1em] d-2-1.north)
        (d-1-2) edge node[above] {$\alpha_2$} (d-1-3)
                     edge node[right] {$\varphi_2$} (d-2-2)
        (d-1-3) edge node[above] {$\alpha_3$} (d-1-4)
                     edge[densely dashed] node[right] {$\varphi_3$} (d-2-3)
        (d-1-4) edge node[above] {$\alpha_4$} (d-1-5)
                     edge[densely dashed] node[right] {$\varphi_4$} (d-2-4)
        (d-1-5) edge node[above] {$\alpha_{n - 2}$} (d-1-6)
        (d-1-6) edge node[above] {$\alpha_{n - 1}$} (d-1-7)
                     edge[densely dashed] node[right] {$\varphi_{n - 1}$} (d-2-6)
        (d-1-7) edge node[above] {$\alpha_n$} (d-1-8)
                     edge[densely dashed] node[right] {$\varphi_n$} (d-2-7)
        ([xshift=-0.1em] d-1-8.south) edge[-] ([xshift=-0.1em] d-2-8.north)
        ([xshift=0.1em] d-1-8.south) edge[-] ([xshift=0.1em] d-2-8.north)
        (d-2-1) edge node[above] {$\beta_1$} (d-2-2)
                     edge node[right] {$\alpha_1$} (d-3-1)
        (d-2-2) edge node[above] {$\beta_2$} (d-2-3)
        ([xshift=-0.1em] d-2-2.south) edge[-] ([xshift=-0.1em] d-3-2.north)
        ([xshift=0.1em] d-2-2.south) edge[-] ([xshift=0.1em] d-3-2.north)
        (d-2-3) edge node[above] {$\beta_3$} (d-2-4)
                     edge[densely dashed] node[right] {$\theta_3$} (d-3-3)
        (d-2-4) edge node[above] {$\beta_4$} (d-2-5)
                     edge[densely dashed] node[right] {$\theta_4$} (d-3-4)
        (d-2-5) edge node[above] {$\beta_{n - 2}$} (d-2-6)
        (d-2-6) edge node[above] {$\beta_{n - 1}$} (d-2-7)
                     edge[densely dashed] node[right] {$\theta_{n - 1}$} (d-3-6)
        (d-2-7) edge node[above] {$\beta_n$} (d-2-8)
                     edge[densely dashed] node[right] {$\theta_n$} (d-3-7)
        (d-2-8) edge node[right] {$\Sigma\alpha_1$} (d-3-8)
        (d-3-1) edge node[above] {$\varphi_2$} (d-3-2)
        (d-3-2) edge node[above] {$\gamma_2$} (d-3-3)
        (d-3-3) edge node[above] {$\gamma_3$} (d-3-4)
        (d-3-4) edge node[above] {$\gamma_4$} (d-3-5)
        (d-3-5) edge node[above] {$\gamma_{n - 2}$} (d-3-6)
        (d-3-6) edge node[above] {$\gamma_{n - 1}$} (d-3-7)
        (d-3-7) edge node[above] {$\gamma_n$} (d-3-8)
        (d-1-4) edge[densely dashed,out=-102,in=30] node[pos=0.15,left] {$\psi_4$}
          (d-3-3)
        (d-1-7) edge[densely dashed,out=-102,in=30] node[pos=0.15,left] {$\psi_n$}
          (d-3-6);
    \end{tikzpicture}
  \end{center}
  with commuting squares and with rows in $\nang$, the dotted
  morphisms exist such that each square commutes, and the
  $n$-$\Sigma$-sequence
  \begin{align*}
    A_3 \xrightarrow{\left[\begin{smallmatrix}
          \alpha_3\\
          \varphi_3
        \end{smallmatrix}\right]}
    & A_4 \oplus B_3 \xrightarrow{\left[\begin{smallmatrix}
          -\alpha_4 & 0\\
          \hfill \varphi_4 & -\beta_3\\
          \hfill \psi_4 & \hfill \theta_3
        \end{smallmatrix}\right]} A_5 \oplus B_4 \oplus C_3
     \xrightarrow{\left[\begin{smallmatrix}
         -\alpha_5 & 0 & 0\\
         -\varphi_5 & -\beta_4 & 0\\
         \hfill \psi_5 & \hfill \theta_4 & \gamma_3
         \end{smallmatrix}\right]}\\
    & \xrightarrow{\left[\begin{smallmatrix}
         -\alpha_5 & 0 & 0\\
         -\varphi_5 & -\beta_4 & 0\\
         \hfill \psi_5 & \hfill \theta_4 & \gamma_3
         \end{smallmatrix}\right]} A_6 \oplus B_5 \oplus C_4
      \xrightarrow{\left[\begin{smallmatrix}
         -\alpha_6 & 0 & 0\\
         \hfill \varphi_6 & -\beta_5 & 0\\
         \hfill \psi_6 & \hfill \theta_5 & \gamma_4
         \end{smallmatrix}\right]}  \cdots\\
    & \cdots \xrightarrow{\left[\begin{smallmatrix}
         -\alpha_{n - 1} & 0 & 0\\
         (-1)^{n - 1} \varphi_{n - 1} & -\beta_{n - 2} & 0\\
         \psi_{n - 1} & \hfill \theta_{n - 2} & \gamma_{n - 3}
         \end{smallmatrix}\right]} A_n \oplus B_{n - 1} \oplus C_{n - 2}
    \xrightarrow{\left[\begin{smallmatrix}
         (-1)^n \varphi_n & -\beta_{n - 1} & 0\\
         \psi_n & \hfill \theta_{n - 1} & \gamma_{n - 2}\\
       \end{smallmatrix}\right]}\\
    & \xrightarrow{\left[\begin{smallmatrix}
         (-1)^n \varphi_n & -\beta_{n - 1} & 0\\
         \psi_n & \hfill \theta_{n - 1} & \gamma_{n - 2}\\
       \end{smallmatrix}\right]} B_n \oplus C_{n - 1}
    \xrightarrow{\left[\begin{smallmatrix}
          \theta_n & \gamma_{n - 1}
        \end{smallmatrix}\right]} C_n \xrightarrow{\Sigma\alpha_2
                     \circ \gamma_n} \Sigma A_3
  \end{align*}
  belongs to $\nang$.
\end{itemize}

In this case, the collection $\nang$ is an \emph{$n$-angulation} of
the category $\C$ (relative to the automorphism $\Sigma$), and the
$n$-$\Sigma$-sequences in $\nang$ are \emph{$n$-angles}. If
$( \C, \Sigma, \nang )$ satisfies (N1), (N2) and (N3), then $\nang$ is
a \emph{pre-$n$-angulation}, and the triple $( \C, \Sigma, \nang )$ is
a \emph{pre-$n$-angulated category}.

\begin{remark}
  \label{rem:Jasso}
  (1) In \cite{GKO}, the fourth defining axiom is the ``mapping cone
  axiom'', which says that in the situation of (N3), the morphisms
  $\varphi_3, \dots, \varphi_n$ can be chosen in such a way that the
  mapping cone of $(\varphi_1, \varphi_2, \dots, \varphi_n)$ belongs
  to $\nang$. When $n = 3$, that is, in the triangulated case, it was
  shown by Neeman in \cite{Neeman1, Neeman2} that this mapping cone
  axiom is equivalent to the octahedral axiom, when the other three
  axioms hold.

  When $n = 3$, the diagram in our axiom (N4) is precisely the
  octahedral diagram, and the axiom is the octahedral axiom. Thus (N4)
  can be thought of as a higher analogue of the octahedral axiom. It
  was shown in \cite{BerghThaule} that when (N1), (N2) and (N3) hold,
  then (N4) and the mapping cone axioms are equivalent.\footnote{We
    thank Gustavo Jasso for showing us the compact diagram we have
    used in axiom (N4). This axiom is not strictly the same as axiom
    (N4*) in \cite{BerghThaule}. However, it follows from the proofs
    in \cite[Section 4]{BerghThaule} that the two are equivalent.}
  Thus our definition of an $n$-angulated category is equivalent to
  that of Geiss, Keller and Oppermann in \cite{GKO}.

  (2) Let $B$ be an object in $\C$. Any $n$-$\Sigma$-sequence
  $A_{\bullet}$ induces an infinite sequence
  \begin{equation*}
    \cdots \to (B, \Sigma^{-1}A_n) \xrightarrow{(\Sigma^{-1}
      \alpha_n)_*} (B,A_1) \xrightarrow{(\alpha_1)_*} (B,A_2) \to \cdots
    \to (B,A_n) \xrightarrow{(\alpha_n)_*} (B, \Sigma A_1) \to \cdots
  \end{equation*}
  of abelian groups and homomorphisms, where $(B,A)$ denotes
  $\Hom_{\C}(B,A)$. If this sequence is exact for all $B$, then
  $A_{\bullet}$ is called an \emph{exact} $n$-$\Sigma$-sequence. By
  \cite[Proposition 2.5]{GKO}, every $n$-angle in a pre-$n$-angulated
  category is exact.

  Now modify axiom (N1) into a new axiom (N1*), whose parts (b) and
  (c) are unchanged, but with the following part (a): if
  $A_\bullet \to B_\bullet$ is a weak isomorphism of exact
  $n$-$\Sigma$-sequences with $A_\bullet \in \nang$, then $B_\bullet$
  belongs to $\nang$. Moreover, let (N2*) be the following weak
  version of axiom (N2): the left rotation of every
  $n$-$\Sigma$-sequenc in $\nang$ also belongs to $\nang$. Then by
  \cite[Theorem 3.4]{BerghThaule}, the following are equivalent:
  \begin{itemize}
  \item[(i)] $( \C, \Sigma, \nang )$ satisfies (N1), (N2), (N3),
  \item[(ii)] $( \C, \Sigma, \nang )$ satisfies (N1*), (N2), (N3),
  \item[(iii)] $( \C, \Sigma, \nang )$ satisfies (N1*), (N2*), (N3).
  \end{itemize}
  It is not known whether only axiom (N2) can be replaced by (N2*),
  without replacing (N1) by (N1*).
\end{remark}

\section{Axiom (N3) is redundant}
\label{sec:(N3)}

In this section we show that the morphism axiom, i.e.\ axiom (N3),
follows from the other axioms. In light of Remark \ref{rem:Jasso}(2),
it would perhaps seem natural to ask which axioms one should use to
prove this. In other words, should one use axioms (N1), (N2), (N4),
axioms (N1*), (N2), (N4), or axioms (N1*), (N2*), (N4)? The answer is
that it does not matter. In fact, axiom (N3) is a consequence of
(N1)(c) and (N4).

\begin{theorem}
  \label{thm:main}
  Axiom \emph{(N3)} follows from \emph{(N1)(c)} and \emph{(N4)}.
\end{theorem}

\begin{proof}
  Assume that we are in the situation of axiom (N3), i.e.\ that we are given the
  solid part of the commutative diagram
  \begin{center}
    \begin{tikzpicture}
      \diagram{d}{2.5em}{2.5em}{
        A_1 & A_2 & A_3 & \cdots & A_n & \Sigma A_1\\
        B_1 & B_2 & B_3 & \cdots & B_n & \Sigma B_1\\
      };
    
      \path[->,midway,font=\scriptsize]
        (d-1-1) edge node[above] {$\alpha_1$} (d-1-2)
                     edge node[right] {$\varphi_1$} (d-2-1)
        (d-1-2) edge node[above] {$\alpha_2$} (d-1-3)
                     edge node[right] {$\varphi_2$} (d-2-2)
        (d-1-3) edge node[above] {$\alpha_3$} (d-1-4)
                     edge[densely dashed] node[right] {$\varphi_3$} (d-2-3)
        (d-1-4) edge node[above] {$\alpha_{n - 1}$} (d-1-5)
        (d-1-5) edge node[above] {$\alpha_n$} (d-1-6)
                     edge[densely dashed] node[right] {$\varphi_n$} (d-2-5)
        (d-1-6) edge node[right] {$\Sigma\varphi_1$} (d-2-6)
        (d-2-1) edge node[above] {$\beta_1$} (d-2-2)
        (d-2-2) edge node[above] {$\beta_2$} (d-2-3)
        (d-2-3) edge node[above] {$\beta_3$} (d-2-4)
        (d-2-4) edge node[above] {$\beta_{n - 1}$} (d-2-5)
        (d-2-5) edge node[above] {$\beta_n$} (d-2-6);
    \end{tikzpicture}
  \end{center}
  with rows in $\nang$. We must prove that dotted morphisms
  $\varphi_3, \dots, \varphi_n$ exist and give a morphism
  $(\varphi_1, \varphi_2, \varphi_3, \dots, \varphi_n )$ of
  $n$-$\Sigma$-sequences.
  
  By assumption, the equality
  $\varphi_2 \circ \alpha_1 = \beta_1 \circ \varphi_1$ holds; we
  define $\gamma_1$ to be this morphism, i.e.
  $\gamma_1 = \varphi_2 \circ \alpha_1 = \beta_1 \circ \varphi_1$. Now
  apply axiom (N1)(c) to the three morphisms
  \begin{equation*}
    \gamma_1 \colon A_1 \to B_2, \hspace{5mm} \varphi_2 \colon A_2 \to
    B_2, \hspace{5mm} \varphi_1 \colon A_1 \to B_1
  \end{equation*}
  in $\C$, and obtain three $n$-$\Sigma$-sequences
  \begin{equation*}
    A_1 \xrightarrow{\gamma_1} B_2 \xrightarrow{\gamma_2}  C_3
    \xrightarrow{\gamma_3} \cdots \xrightarrow{\gamma_{n - 1}} C_n
    \xrightarrow{\gamma_n} \Sigma A_1,
  \end{equation*}
  \begin{equation*}
    A_2 \xrightarrow{\varphi_2} B_2 \xrightarrow{\delta_2}  D_3
    \xrightarrow{\delta_3} \cdots \xrightarrow{\delta_{n - 1}} D_n
    \xrightarrow{\delta_n} \Sigma A_2
  \end{equation*}
  and
  \begin{equation*}
    A_1 \xrightarrow{\varphi_1} B_1 \xrightarrow{\epsilon_2}  E_3
    \xrightarrow{\epsilon_3} \cdots \xrightarrow{\epsilon_{n - 1}} E_n
    \xrightarrow{\epsilon_n} \Sigma A_1
  \end{equation*}
  in $\nang$. Next, consider the two diagrams
  \begin{center}
    \begin{tikzpicture}
      \diagram{d}{2.5em}{2.5em}{
        A_1 & A_2 & A_3 & A_4 & \cdots & A_{n - 1} & A_n & \Sigma
        A_1\\
        A_1 & B_2 & C_3 & C_4 & \cdots & C_{n - 1} & C_n & \Sigma
        A_1\\
        A_2 & B_2 & D_3 & D_4 & \cdots & D_{n - 1} & D_n & \Sigma
        A_2\\
      };

      \path[->,midway,font=\scriptsize]
        (d-1-1) edge node[above] {$\alpha_1$} (d-1-2)
        ([xshift=-0.1em] d-1-1.south) edge[-] ([xshift=-0.1em] d-2-1.north)
        ([xshift=0.1em] d-1-1.south) edge[-] ([xshift=0.1em] d-2-1.north)
        (d-1-2) edge node[above] {$\alpha_2$} (d-1-3)
                     edge node[right] {$\varphi_2$} (d-2-2)
        (d-1-3) edge node[above] {$\alpha_3$} (d-1-4)
                     edge[densely dashed] node[right] {$\rho_3$} (d-2-3)
        (d-1-4) edge node[above] {$\alpha_4$} (d-1-5)
                     edge[densely dashed] node[right] {$\rho_4$} (d-2-4)
        (d-1-5) edge node[above] {$\alpha_{n - 2}$} (d-1-6)
        (d-1-6) edge node[above] {$\alpha_{n - 1}$} (d-1-7)
                     edge[densely dashed] node[right] {$\rho_{n - 1}$} (d-2-6)
        (d-1-7) edge node[above] {$\alpha_n$} (d-1-8)
                     edge[densely dashed] node[right] {$\rho_n$} (d-2-7)
        ([xshift=-0.1em] d-1-8.south) edge[-] ([xshift=-0.1em] d-2-8.north)
        ([xshift=0.1em] d-1-8.south) edge[-] ([xshift=0.1em] d-2-8.north)
        (d-2-1) edge node[above] {$\gamma_1$} (d-2-2)
                     edge node[right] {$\alpha_1$} (d-3-1)
        (d-2-2) edge node[above] {$\gamma_2$} (d-2-3)
        ([xshift=-0.1em] d-2-2.south) edge[-] ([xshift=-0.1em] d-3-2.north)
        ([xshift=0.1em] d-2-2.south) edge[-] ([xshift=0.1em] d-3-2.north)
        (d-2-3) edge node[above] {$\gamma_3$} (d-2-4)
                     edge[densely dashed] node[right] {$\theta_3$} (d-3-3)
        (d-2-4) edge node[above] {$\gamma_4$} (d-2-5)
                     edge[densely dashed] node[right] {$\theta_4$} (d-3-4)
        (d-2-5) edge node[above] {$\gamma_{n - 2}$} (d-2-6)
        (d-2-6) edge node[above] {$\gamma_{n - 1}$} (d-2-7)
                     edge[densely dashed] node[right] {$\theta_{n - 1}$} (d-3-6)
        (d-2-7) edge node[above] {$\gamma_n$} (d-2-8)
                     edge[densely dashed] node[right] {$\theta_n$} (d-3-7)
        (d-2-8) edge node[right] {$\Sigma\alpha_1$} (d-3-8)
        (d-3-1) edge node[above] {$\varphi_2$} (d-3-2)
        (d-3-2) edge node[above] {$\delta_2$} (d-3-3)
        (d-3-3) edge node[above] {$\delta_3$} (d-3-4)
        (d-3-4) edge node[above] {$\delta_4$} (d-3-5)
        (d-3-5) edge node[above] {$\delta_{n - 2}$} (d-3-6)
        (d-3-6) edge node[above] {$\delta_{n - 1}$} (d-3-7)
        (d-3-7) edge node[above] {$\delta_n$} (d-3-8);
    \end{tikzpicture}
  \end{center}
  and
  \begin{center}
    \begin{tikzpicture}
      \diagram{d}{2.5em}{2.5em}{
        A_1 & B_1 & E_3 & E_4 & \cdots & E_{n - 1} & E_n & \Sigma A_1\\
        A_1 & B_2 & C_3 & C_4 & \cdots & C_{n - 1} & C_n & \Sigma
        A_1\\
        B_1 & B_2 & B_3 & B_4 & \cdots & B_{n - 1} & B_n & \Sigma
        B_1\\
      };

      \path[->,midway,font=\scriptsize]
        (d-1-1) edge node[above] {$\varphi_1$} (d-1-2)
        ([xshift=-0.1em] d-1-1.south) edge[-] ([xshift=-0.1em] d-2-1.north)
        ([xshift=0.1em] d-1-1.south) edge[-] ([xshift=0.1em] d-2-1.north)
        (d-1-2) edge node[above] {$\epsilon_2$} (d-1-3)
                     edge node[right] {$\beta_1$} (d-2-2)
        (d-1-3) edge node[above] {$\epsilon_3$} (d-1-4)
                     edge[densely dashed] node[right] {$\psi_3$} (d-2-3)
        (d-1-4) edge node[above] {$\epsilon_4$} (d-1-5)
                     edge[densely dashed] node[right] {$\psi_4$} (d-2-4)
        (d-1-5) edge node[above] {$\epsilon_{n - 2}$} (d-1-6)
        (d-1-6) edge node[above] {$\epsilon_{n - 1}$} (d-1-7)
                     edge[densely dashed] node[right] {$\psi_{n - 1}$} (d-2-6)
        (d-1-7) edge node[above] {$\epsilon_n$} (d-1-8)
                     edge[densely dashed] node[right] {$\psi_n$} (d-2-7)
        ([xshift=-0.1em] d-1-8.south) edge[-] ([xshift=-0.1em] d-2-8.north)
        ([xshift=0.1em] d-1-8.south) edge[-] ([xshift=0.1em] d-2-8.north)
        (d-2-1) edge node[above] {$\gamma_1$} (d-2-2)
                     edge node[right] {$\varphi_1$} (d-3-1)
        (d-2-2) edge node[above] {$\gamma_2$} (d-2-3)
        ([xshift=-0.1em] d-2-2.south) edge[-] ([xshift=-0.1em] d-3-2.north)
        ([xshift=0.1em] d-2-2.south) edge[-] ([xshift=0.1em] d-3-2.north)
        (d-2-3) edge node[above] {$\gamma_3$} (d-2-4)
                     edge[densely dashed] node[right] {$\eta_3$} (d-3-3)
        (d-2-4) edge node[above] {$\gamma_4$} (d-2-5)
                     edge[densely dashed] node[right] {$\eta_4$} (d-3-4)
        (d-2-5) edge node[above] {$\gamma_{n - 2}$} (d-2-6)
        (d-2-6) edge node[above] {$\gamma_{n - 1}$} (d-2-7)
                     edge[densely dashed] node[right] {$\eta_{n - 1}$} (d-3-6)
        (d-2-7) edge node[above] {$\gamma_n$} (d-2-8)
                     edge[densely dashed] node[right] {$\eta_n$} (d-3-7)
        (d-2-8) edge node[right] {$\Sigma\varphi_1$} (d-3-8)
        (d-3-1) edge node[above] {$\beta_1$} (d-3-2)
        (d-3-2) edge node[above] {$\beta_2$} (d-3-3)
        (d-3-3) edge node[above] {$\beta_3$} (d-3-4)
        (d-3-4) edge node[above] {$\beta_4$} (d-3-5)
        (d-3-5) edge node[above] {$\beta_{n - 2}$} (d-3-6)
        (d-3-6) edge node[above] {$\beta_{n - 1}$} (d-3-7)
        (d-3-7) edge node[above] {$\beta_n$} (d-3-8);
    \end{tikzpicture}
  \end{center}
  Note that all the rows in these two diagrams are
  $n$-$\Sigma$-sequences in $\nang$. Moreover, it follows from the
  definition of the morphism $\gamma_1$ that the solid parts of the
  diagrams commute. We may therefore apply axiom (N4); there exist
  dotted morphisms such that each square commutes. Axiom (N4) also
  gives other morphisms, the diagonal ones, but these are not needed
  here.
  
  The two diagrams share the same middle row. We may therefore combine
  the upper part of the first diagram and the lower part of the
  second, and obtain the commutative diagram
  \begin{center}
    \begin{tikzpicture}
      \diagram{d}{2.5em}{2.5em}{
        A_1 & A_2 & A_3 & A_4 & \cdots & A_{n - 1} & A_n & \Sigma A_1\\
        A_1 & B_2 & C_3 & C_4 & \cdots & C_{n - 1} & C_n & \Sigma
        A_1\\
                B_1 & B_2 & B_3 & B_4 & \cdots & B_{n - 1} & B_n & \Sigma
        B_1\\
      };

      \path[->,midway,font=\scriptsize]
        (d-1-1) edge node[above] {$\alpha_1$} (d-1-2)
        ([xshift=-0.1em] d-1-1.south) edge[-] ([xshift=-0.1em] d-2-1.north)
        ([xshift=0.1em] d-1-1.south) edge[-] ([xshift=0.1em] d-2-1.north)
        (d-1-2) edge node[above] {$\alpha_2$} (d-1-3)
                     edge node[right] {$\varphi_2$} (d-2-2)
        (d-1-3) edge node[above] {$\alpha_3$} (d-1-4)
                     edge node[right] {$\rho_3$} (d-2-3)
        (d-1-4) edge node[above] {$\alpha_4$} (d-1-5)
                     edge node[right] {$\rho_4$} (d-2-4)
        (d-1-5) edge node[above] {$\alpha_{n - 2}$} (d-1-6)
        (d-1-6) edge node[above] {$\alpha_{n - 1}$} (d-1-7)
                     edge node[right] {$\rho_{n - 1}$} (d-2-6)
        (d-1-7) edge node[above] {$\alpha_n$} (d-1-8)
                     edge node[right] {$\rho_n$} (d-2-7)
        ([xshift=-0.1em] d-1-8.south) edge[-] ([xshift=-0.1em] d-2-8.north)
        ([xshift=0.1em] d-1-8.south) edge[-] ([xshift=0.1em] d-2-8.north)
        (d-2-1) edge node[above] {$\gamma_1$} (d-2-2)
                     edge node[right] {$\varphi_1$} (d-3-1)
        (d-2-2) edge node[above] {$\gamma_2$} (d-2-3)
        ([xshift=-0.1em] d-2-2.south) edge[-] ([xshift=-0.1em] d-3-2.north)
        ([xshift=0.1em] d-2-2.south) edge[-] ([xshift=0.1em] d-3-2.north)
        (d-2-3) edge node[above] {$\gamma_3$} (d-2-4)
                     edge node[right] {$\eta_3$} (d-3-3)
        (d-2-4) edge node[above] {$\gamma_4$} (d-2-5)
                     edge node[right] {$\eta_4$} (d-3-4)
        (d-2-5) edge node[above] {$\gamma_{n - 2}$} (d-2-6)
        (d-2-6) edge node[above] {$\gamma_{n - 1}$} (d-2-7)
                     edge node[right] {$\eta_{n - 1}$} (d-3-6)
        (d-2-7) edge node[above] {$\gamma_n$} (d-2-8)
                     edge node[right] {$\eta_n$} (d-3-7)
        (d-2-8) edge node[right] {$\Sigma\varphi_1$} (d-3-8)
        (d-3-1) edge node[above] {$\beta_1$} (d-3-2)
        (d-3-2) edge node[above] {$\beta_2$} (d-3-3)
        (d-3-3) edge node[above] {$\beta_3$} (d-3-4)
        (d-3-4) edge node[above] {$\beta_4$} (d-3-5)
        (d-3-5) edge node[above] {$\beta_{n - 2}$} (d-3-6)
        (d-3-6) edge node[above] {$\beta_{n - 1}$} (d-3-7)
        (d-3-7) edge node[above] {$\beta_n$} (d-3-8);
    \end{tikzpicture}
  \end{center}
  Finally, by omitting the middle row, we obtain the commutative
  diagram
  \begin{center}
    \begin{tikzpicture}
      \diagram{d}{2.5em}{2.5em}{
        A_1 & A_2 & A_3 & A_4 & \cdots & A_{n - 1} & A_n & \Sigma
        A_1\\ 
        B_1 & B_2 & B_3 & B_4 & \cdots & B_{n - 1} & B_n & \Sigma
        B_1\\ 
      };

      \path[->,midway,font=\scriptsize]
        (d-1-1) edge node[above] {$\alpha_1$} (d-1-2)
                     edge node[right] {$\varphi_1$} (d-2-1)
        (d-1-2) edge node[above] {$\alpha_2$} (d-1-3)
                     edge node[right] {$\varphi_2$} (d-2-2)
        (d-1-3) edge node[above] {$\alpha_3$} (d-1-4)
                     edge node[right] {$\eta_3 \circ \rho_3$} (d-2-3)
        (d-1-4) edge node[above] {$\alpha_4$} (d-1-5)
                     edge node[right] {$\eta_4 \circ \rho_4$} (d-2-4)
        (d-1-5) edge node[above] {$\alpha_{n - 2}$} (d-1-6)
        (d-1-6) edge node[above] {$\alpha_{n - 1}$} (d-1-7)
                     edge node[right] {$\eta_{n-1} \circ \rho_{n - 1}$} (d-2-6)
        (d-1-7) edge node[above] {$\alpha_n$} (d-1-8)
                     edge node[right] {$\eta_n \circ \rho_n$} (d-2-7)
        (d-1-8) edge node[right] {$\Sigma\varphi_1$} (d-2-8)
        (d-2-1) edge node[above] {$\beta_1$} (d-2-2)
        (d-2-2) edge node[above] {$\beta_2$} (d-2-3)
        (d-2-3) edge node[above] {$\beta_3$} (d-2-4)
        (d-2-4) edge node[above] {$\beta_4$} (d-2-5)
        (d-2-5) edge node[above] {$\beta_{n - 2}$} (d-2-6)
        (d-2-6) edge node[above] {$\beta_{n - 1}$} (d-2-7)
        (d-2-7) edge node[above] {$\beta_n$} (d-2-8);
    \end{tikzpicture}
  \end{center}
  Now for each $3\leq k \leq n$, define a morphism $\varphi_k$ by
  $\varphi_k = \eta_k \circ \rho_k$. Then
  $(\varphi_1, \varphi_2, \varphi_3, \dots, \varphi_n)$ is a morphism
  of $n$-$\Sigma$-sequences, and this completes the proof.
\end{proof}

\begin{remark}
  \label{rem:May}
  (1) As noted in the proof, not all of axiom (N4) is needed. The
  diagonal morphisms provided by the axiom play no role, and,
  consequently, neither does the $n$-$\Sigma$-sequence involving these
  morphisms.
  
  (2) Note that if $n = 3$, that is, in the triangulated case, the
  result recovers that of May in \cite{May}, but with a slightly
  different setup. Indeed, May's result was the inspiration for the
  theorem.
\end{remark}

\bibliographystyle{plain}

%
% End document
%
\end{document}